\documentclass[12pt]{article}
\usepackage{geometry}                
\usepackage[parfill]{parskip}    
\usepackage{graphicx}
\usepackage{amssymb}
\usepackage{epstopdf}
\usepackage{amsmath}
\usepackage{amsthm}
\usepackage{dsfont}
\newtheorem{thm}{Theorem}[section]
\newtheorem{lem}[thm]{Lemma}

\newtheorem{cor}[thm]{Corollary}
\newtheorem{defn}[thm]{Definition}

\newtheorem{innercustomthm}{Theorem}
\newenvironment{customthm}[1]
  {\renewcommand\theinnercustomthm{#1}\innercustomthm}
  {\endinnercustomthm}

\newcommand{\F}{\mathcal{F}}
\newcommand{\G}{G_{\F}}
\def\sm{{\backslash}}

\title{Covering hypergraphs are eulerian}
\author{Mateja \v{S}ajna\footnote{Corresponding author. Email: msajna@uottawa.ca. Mailing address: Department of Mathematics and Statistics, University of Ottawa, 150 Louis-Pasteur Private, Ottawa, ON, K1N 6N5, Canada.} \; and Andrew Wagner \\ {\small University of Ottawa}}

\begin{document}

\maketitle
\begin{abstract} An {\em Euler tour} in a hypergraph (also called a {\em rank-2 universal cycle} or {\em 1-overlap cycle} in the context of designs) is a closed walk that traverses every edge exactly once.  In this paper, we define a {\em covering k-hypergraph} to be a non-empty $k$-uniform hypergraph in which every $(k-1)$-subset of vertices appear together in at least one edge.  We then show that every covering $k$-hypergraph, for $k\geq 3$, admits an Euler tour if and only if it has at least two edges.

\end{abstract}


\baselineskip 17.8pt

\section{Introduction}

An Euler tour of a graph or hypergraph is a closed walk that traverses every edge exactly once.  The complete characterization of graphs that admit an Euler tour, which was conjectured by Euler in 1741 \cite{E}, and proved by Hierholzer and Wiener in 1873 \cite{HW}, is a fundamental and accessible result in graph theory.  However, analogous results for hypergraphs have only recently begun to be explored.

To our knowledge, Lonc and Naroski were the first to extend the notion of an Euler tour to hypergraphs.  In a 2010 paper  \cite{LN} they showed that the problem of existence of an Euler tour is NP-complete on the set of $k$-uniform hypergraphs, for any $k\geq 3$, as well as when restricted to a particular subclass of 3-uniform hypergraphs.

A {\em Steiner triple} ({\em quadruple}) {\em system} is a 3-uniform (4-uniform, respectively) hypergraph in which every pair (triple, respectively) of vertices lie together in exactly one edge. Initial results for Euler tours in Steiner triple and quadruple systems used different terminology. Namely, in 2012, Dewar and Stevens \cite{DS} proved that every cyclic Steiner triple system of order greater than three,  and every cyclic twofold triple system admits a {\em rank-2 universal cycle} (that is, Euler tour), and in 2013, Horan and Hurlbert showed that for every admissible order greater than four, there exist a Steiner triple system \cite{HH2} and a Steiner quadruple system \cite{HH1} with a {\em 1-overlap cycle} (that is, Euler tour).

We define a {\em covering $k$-hypergraph}, for $k\geq 3$, to be a non-empty $k$-uniform hypergraph in which every $(k-1)$-subset of vertices appear together in at least one edge.

In a 2015 paper \cite{BS2}, Bahmanian and \v{S}ajna  embarked on a systematic study of eulerian properties of general hypergraphs; some of their techniques and results will be used in this paper. In particular, they introduced the notion of an {\em Euler family} --- a collection of closed trails that jointly traverse each edge exactly once --- and showed that the problem of existence of an Euler family is polynomial on the class of all hypergraphs. In addition, they proved that every covering 3-hypergraph with at least two edges admits an Euler family.

Most recently, the present authors gave a short proof \cite{SW} to show that every triple system --- that is, a 3-uniform hypergraph in which every pair of vertices lie together in the same number of edges --- admits an Euler tour as long as it has at least two edges.

The main result of this paper is as follows. \\

\begin{thm}\label{thm:main} Let  $H$ be a covering $k$-hypergraph, for $k\geq 3$.  Then $H$ admits an Euler tour if and only if it has at least two edges.
\end{thm}

Thus, Theorem~\ref{thm:main} represents a significant improvement on \cite{SW}, as well as on the above-mentioned result of \cite{BS2}.

The framework for the proof of Theorem~\ref{thm:main} is presented in Section~\ref{sec:mainresult}, while the technical details are given in Section~\ref{sec:mainlemmas}. Here, we also introduce our main, new technique of interchanging cycles, which resembles the technique of alternating paths in matching theory.

\section{Preliminaries}\label{sec:preliminaries}

We use hypergraph terminology established in \cite{BS1, BS2}, which applies to loopless graphs as well.  Any graph theory terms not explained here can be found in \cite{BM}.

A {\em hypergraph} $H$ is a pair $(V,E)$, where $V$ is a non-empty set, and $E$ is a multiset of elements from $2^V$.  The elements of $V=V(H)$ and $E=E(H)$ are called the {\em vertices} and {\em edges} of $H$, respectively. The {\em order} of $H$ is $|V|$.  A hypergraph of order 1 is called {\em trivial}, and a hypergraph with no edges is called {\em empty}.

Distinct vertices $u$ and $v$ in a hypergraph $H=(V,E)$ are called {\em adjacent} (or {\em neighbours})  if they lie in the same edge, while a vertex $v$ and an edge $e$ are said to be {\em incident} if $v\in e$.    The {\em degree} of $v$ in $H$, denoted $\deg_H(v)$, is the number of edges of $H$ incident with $v$.  A hypergraph $H$ is called {\em $k$-uniform} if every edge of $H$ has cardinality $k$.

A {\em covering $k$-hypergraph}, for $k\geq 3$, is a non-empty $k$-uniform hypergraph with the property that every $(k-1)$-subset of vertices lie together in at least one edge.

The {\em incidence graph} of a hypergraph $H=(V,E)$ is a bipartite simple graph $\mathcal{G}(H)$ with vertex set $V \cup E$ and bipartition $\{V,E\}$ such that
vertices $v\in V$ and $e\in E$ of $\mathcal{G}(H)$ are adjacent if and only if $v$ is incident with $e$ in $H$.  The elements of $V$ and $E$ are called {\em v-vertices} and {\em e-vertices} of $\mathcal{G}(H)$, respectively.

A hypergraph $H' = (V',E')$ is called a {\em subhypergraph} of the hypergraph $H = (V,E)$ if $V'\subseteq V$ and $E'=\{ e \cap V': e \in E''\}$ for some submultiset $E''$ of $E$. For any subset $V'\subseteq V$, we define the {\em subhypergraph of $H$ induced by} $V'$, denoted $H[V']$,  to be the hypergraph $(V',E')$ with $E' = \{e\cap V': e\in E, e\cap V'\ne\emptyset \}$. Thus, we obtain the subhypergraph induced by $V'$  by deleting all vertices in $V \sm V'$ from $V$ and from each edge of $H$, and subsequently deleting all empty edges.
By $H - V'$ we denote  the subhypergraph of $H$ induced by $V \sm V'$, and for $v \in V$, we write shortly $H - v$ instead of $H - \{ v\}$.
For any subset $E'\subseteq E$, we denote the subhypergraph $(V, E \sm E')$ of $H$ by $H \sm E'$, and for $e \in E$, we write $H  \sm  e$ instead of $H \sm  \{ e\}$. For any multiset $E'$ of $2^V$, the symbol $H+E'$ will denote the hypergraph obtained from $H$ by adjoining all edges in $E'$.

A {\em $v_0v_k$-walk} in $H$ is a sequence  $W=v_0e_1v_1e_2 \ldots e_kv_k$ such that $v_0,\ldots , v_k\in V$; $e_1, \ldots , e_k\in E$; and  $v_{i-1},v_i\in e_i$ with $v_{i-1}\neq v_i$ for all $i=1, \ldots, k$.  A walk is said to {\em traverse} each of the vertices and edges in the sequence.  The vertices $v_0,v_1, \ldots, v_k$ are called the {\em anchors} of $W$.  If $e_1, e_2, \ldots, e_k$ are pairwise distinct, then $W$ is called a {\em trail} ({\em strict trail} in \cite{BS1,BS2}); if $v_0=v_k$ and $k \ge 2$, then $W$ is {\em closed}.  If $W$ is a closed trail and $v_0,v_1,\ldots , v_{k-1}$ are pairwise distinct, then $W$ is called a {\em cycle}; more specifically, a cycle {\em of length $k$}, or {\em $k$-cycle}.

A hypergraph $H$ is {\em connected} if, for any pair $u,v\in V(H)$, there exists a $uv$-walk in $H$.  A {\em connected component} is a maximal connected subhypergraph of $H$ without empty edges.  We call $v \in V(H)$ a {\em cut vertex} of $H$, and  $e\in E(H)$ a {\em cut edge} of $H$, if $H-v$ and $H\sm  e$, respectively, have more connected components than $H$.

An {\em Euler family} of a hypergraph $H$ is a collection of pairwise anchor-disjoint and edge-disjoint closed trails that jointly traverse every edge of $H$, and an {\em Euler tour}  is a closed trail that traverses every edge of $H$. Thus, an Euler family of cardinality 1 corresponds to an Euler tour. The  closed trails in an Euler family $\F$ are called the {\em components} of $\F$, and an Euler family of $H$ with the smallest number of components is called {\em minimum.}
A hypergraph that is either empty or admits an Euler tour (family) is called {\em eulerian (quasi-eulerian)}.

\section{Proof of the main result}\label{sec:mainresult}

All of our work in Section~\ref{sec:mainlemmas} will serve to prove the following result. \\

\begin{thm}\label{thm:covering} Let $H$ be a covering 3-hypergraph  with at least two edges.  Then $H$ is eulerian.
\end{thm}

\begin{proof} Let $n$ be the order of $H$.
If $n\geq 7$, then $H$ is eulerian by Theorem~\ref{thm:order7}, and if $3\leq n\leq 6$, then $H$ is eulerian by Lemma~\ref{lem:coveringorder3-5}.
\end{proof}

Theorem~\ref{thm:covering} will now be used as the basis of induction in the proof of our main result, which we restate below.\\

\begin{customthm}{\ref{thm:main} }
Let $H$ be a covering $k$-hypergraph, for $k\geq 3$.  Then $H$ is eulerian if and only if it has at least two edges.
\end{customthm}

\begin{proof}
By definition, we know $H$ is non-empty. The necessity is then clear.

To prove sufficiency, we use induction on $k$.  For $k=3$, Theorem~\ref{thm:covering} implies that $H$ is eulerian.  Suppose that, for some fixed $k\geq 3$, the statement holds: that is, any covering $k$-hypergraph with at least two edges is eulerian.

Let $H=(V,E)$ be a covering $(k+1)$-hypergraph with $|E|\geq 2.$  Fix any $v\in V$, and let $\mathcal{B}=\{e\sm \{v\}:e\in E, v\in e\}$ and $E'=\{e\in E: v\not\in e\}$.  Then let $\mathcal{C}$ be a set  obtained from $E'$ by removing an arbitrary vertex from each $e\in E'$.  Let $H^*=(V\sm \{v\},\mathcal{B}\cup\mathcal{C}),$ where $\mathcal{B}\cup\mathcal{C}$ denotes the multiset union of $\mathcal{B}$ and $\mathcal{C}$.  Then $H^*$ is a covering $k$-hypergraph.  Since $|E(H^*)|=|E(H)|\geq 2$, by the induction hypothesis, we have that $H^*$ admits an Euler tour $T$. Replacing each edge in $T$ by the corresponding edge of $H$, we obtain an Euler tour of $H$.

The result follows by induction.
\end{proof}

\section{Covering 3-hypergraphs}\label{sec:mainlemmas}

\subsection{The big tools}\label{sec:big-tools}

The following theorem will allow us to translate our analysis from a hypergraph to its incidence graph.\\

\begin{thm}\label{thm:incidencegraph}{\rm \cite[Theorem 2.18]{BS2}} Let $H$ be a hypergraph and $G$ its incidence graph.  Then the following hold.
\begin{description}
\item[(1)] $H$ is quasi-eulerian if and only if $G$ has a spanning subgraph $G'$ such that $\text{deg}_{G'}(e) = 2$ for all $e\in E(H)$, and $\text{deg}_{G'}(v)$ is even for all $v\in V(H)$.
\item[(2)] $H$ is eulerian if and only if $G$ has a spanning subgraph $G'$ with at most one non-trivial connected component
    such that $\text{deg}_{G'}(e) = 2$ for all $e\in E(H)$, and $\text{deg}_{G'}(v)$ is even for all $v\in V(H)$.\qed
\end{description}
\end{thm}

\phantom{aaa}

The proof of Theorem~\ref{thm:incidencegraph} reveals that each edge $ve$ of $G'$ corresponds to a sequence $ve$ or $ev$ in a component of the Euler family $\F$. Consequently, the components of $\F$ are in bijective correspondence with the non-trivial connected components of $G'$. We will henceforth call $G'$ the {\em subgraph of $G$ corresponding to $\F$}, and denote it by $\G$. An edge $e\in E(G)$ will be called a  {\em $\G$-edge} if $e\in E(\G)$, and a {\em non-$\G$-edge} otherwise.

We are now prepared to introduce the main tool that will be used throughout this section.
Note that we define the {\em symmetric difference} of simple graphs  $G_1=(V_1,E_1)$ and $G_2=(V_2,E_2)$, denoted $G_1\Delta G_2$, as the graph $(V_1\cup V_2, E_1\Delta E_2)$, where $E_1\Delta E_2$ is the symmetric difference of $E_1$ and $E_2$. \\

\begin{figure}[t]
\centerline{\includegraphics[scale=0.8]{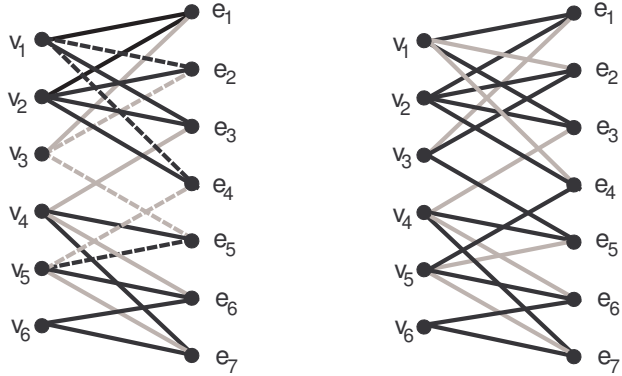}}
\caption{The incidence graph of a 3-uniform hypegraph with subgraphs $\G$ (left) and $G_{\F'}=\G\Delta C$ (right). Black (grey) edges are those in (not in) the subgraph corresponding to the Euler family, and dashed edges represent the $\F$-interchanging (in fact, $\F$-diminishing) cycle $C$. }\label{fig:F-IntCycle}
\end{figure}

\begin{defn}{\rm Let $H$ be a hypergraph with an Euler family $\F$ and incidence graph $G$, and let $\G$ be the subgraph of $G$ corresponding to $\F$.

A cycle $C$ of $G$ is called {\em $\F$-interchanging} if $\deg_{\G\cap C}(e)=1$ for all e-vertices $e\in E(H)\cap V(C)$. If $C$ is an $\F$-interchanging cycle of $G$ and $\G\Delta C$ has fewer non-trivial connected components than $\G$, then $C$ is an {\em $\F$-diminishing cycle} of $G$.
}
\end{defn}

See Figure~\ref{fig:F-IntCycle} for an example of an $\F$-interchanging (in fact, $\F$-diminishing) cycle. Note that in all figures, the incidence graph of a hypergraph will be drawn with v-vertices on the left.\\

\begin{lem}\label{lem:interchange} Let $H$ be a hypergraph with an Euler family $\F$ and incidence graph $G$, and let $\G$ be the subgraph of $G$ corresponding to $\F$.

If $C$ is an $\F$-interchanging cycle of $G$, then $\G\Delta C$ is a subgraph of $G$ corresponding to an Euler family of $H$.
\end{lem}

\begin{proof}For all $x\in V(G)$, we have that $$\deg_{\G\Delta C}(x) = \deg_{\G}(x) + \deg_C(x) - 2\deg_{\G\cap C}(x),$$ so it is clear that in $\G\Delta C$, every v-vertex has even degree, and every e-vertex has degree 2.  Hence, by Theorem~\ref{thm:incidencegraph}, the subgraph $\G\Delta C$ corresponds to an Euler family of $H$.
\end{proof}

The general approach to proving Theorem~\ref{thm:covering} will be as follows: we start with an Euler family $\F$ of $H$, which is guaranteed to exist by Theorem~\ref{cor:quasieuleriancovering} below, and then use an $\F$-diminishing cycle to construct an Euler family with fewer components. \\

\begin{thm}\label{cor:quasieuleriancovering}{\rm \cite[Corollary 5.3]{BS2}} Let $H$ be a covering 3-hypergraph with at least two edges.  Then $H$ is quasi-eulerian.\qed
\end{thm}

\subsection{Technical lemmas}\label{sec:technicallemmas}

In this section, we state the technical lemmas needed to prove Theorem~\ref{thm:covering}. \\

\begin{lem}\label{prop:noncut} Let $G$ be a connected loopless graph.  If $G$ has a block of order $k \ge 2$, then it has at least $k$ vertices that are not cut vertices.
\end{lem}

\begin{proof}
Let ${\cal B}$ be the block tree of $G$; that is, the bipartite graph whose vertices are the blocks and the cut vertices of $G$, with $v$ and $B$ adjacent in ${\cal B}$ if and only if $v$ is a cut vertex of $G$ and $B$ is a block of $G$ containing $v$. Choose a block $B$ of order $k$ to be the root of the tree  ${\cal B}$, and observe that any leaf of ${\cal B}$ must be a block. If $v \in V(B)$ is a cut vertex of $G$, consider any block $B'$ that is a leaf in the subtree of ${\cal B}$ rooted at $v$. Since $G$ is loopless, $B'$ has at least two vertices, and since $B'$ is a leaf in ${\cal B}$, only one of them is a cut vertex of $G$. Hence each $v \in V(B)$ is either not a cut vertex of $G$, or else the subtree rooted at $v$ contains a vertex that is not a cut vertex of $G$. The result follows.
\end{proof}

For the following four results --- Corollary~\ref{cor:noncut} through Corollary~\ref{cor:dimcycle} --- we assume that $H$ is an arbitrary hypergraph (not necessarily covering or 3-uniform) with an Euler family $\F$ and incidence graph $G$, and  $\G$ is the subgraph of $G$ corresponding to $\F$. \\

\begin{cor}\label{cor:noncut} Let $G_1$ be a non-trivial connected component of $\G$.  Then at least two v-vertices of $G_1$ are not cut vertices of $\G$.  Furthermore, if $G_1$ has a cycle of length $2k$, then at least $k$ v-vertices of $G_1$ are not cut vertices of $\G$.
\end{cor}

\begin{proof} Since the e-vertices of $\G$ each have degree 2, observe that $\G$ is the incidence graph of a loopless graph (2-uniform hypergraph)  $H$, and $G_1$ is the incidence graph of a non-trivial connected component $H_1$ of $H$.  A cycle $C$ of $G_1$ corresponds to a cycle of $H_1$ whose anchors are the v-vertices of $C$ \cite[Lemma 3.6]{BS1}, and cut vertices of $G_1$ that are v-vertices correspond to cut vertices of $H_1$ \cite[Theorem 3.23]{BS1}.  Since $H_1$ has a block of order at least 2, and any cycle of $H_1$ is contained in a block, the result follows by applying Lemma~\ref{prop:noncut} to $H_1$.
\end{proof}

\phantom{a}

\begin{lem}\label{lem:dimcycle} Let $C$ be an $\F$-interchanging cycle in $G$, and let  $G_1, \ldots, G_k$ be all of the connected components of $\G$ that contain the vertices of $C$. Assume that $G_i\sm  E(C)$, for each $i=1, \ldots, k$, has at most one non-trivial connected component. Then $G^*=(\G\Delta C)[V(G_1)\cup \ldots \cup V(G_k)]$ has at most one non-trivial connected component.
\end{lem}

\begin{proof} Write $C$ as
$$C=u_0W_0v_0u_1W_1v_1\ldots u_{\ell-1}W_{\ell-1}v_{\ell-1}u_0,$$ where each $W_i$, for $i\in\mathds{Z}_{\ell}$, is a maximal subwalk of $C$ whose vertices lie in a single connected component of $\G$, and
$u_i$ and $v_i$ are its endpoints. Let $G_{j_i}$ be the connected component of $\G$ containing the $u_iv_i$-walk $W_i$, and suppose $G_{j_i}$ is non-trivial. Then $u_i$ and $v_i$ cannot be isolated in $G_{j_i}$. If they were isolated in $G_{j_i}\sm  E(C)$, then each would be incident with two $\G$-edges of $C$ because $G_{j_i}$ is an even graph ---  a contradiction. Hence $u_i$ and $v_i$ lie in the unique non-trivial connected component of $G_{j_i}\sm  E(C)$.

Thus $C$ traverses vertices in all of the non-trivial connected components of $G_i\sm  E(C)$, for $i=1,\ldots,k$. Since the edges $v_0u_1, v_1u_2, \ldots, v_{\ell-1}u_0$ link these non-trivial components in a circular fashion, $G^*=\bigcup_{i=1}^k (G_i\sm  E(C) ) + \{ v_iu_{i+1}: i \in \mathds{Z}_{\ell} \}$ has at most one non-trivial connected component.
\end{proof}

\phantom{a}

\begin{cor}\label{cor:simpledimcycle}
Let $C$ be an $\F$-interchanging cycle of $G$.  If the $\G$-edges of $C$ lie in distinct connected components of $\G$, then $C$ is an $\F$-diminishing cycle.
\end{cor}

\begin{proof}
Since $C$ has at least two $\G$-edges, it traverses vertices from at least two non-trivial connected components of $\G$.
Each such component $G_i$ is an even graph with $|E(G_i)\cap E(C)|\leq 1$, so $G_i\sm  E(C)$ is connected.  Applying Lemma~\ref{lem:dimcycle}, we conclude that the number of non-trivial connected components of $\G\Delta C$ is less than that of $\G$, so $C$ is an $\F$-diminishing cycle.
\end{proof}

\phantom{a}

\begin{cor}\label{cor:dimcycle} Let $C$ be an $\F$-interchanging cycle in $G$.
\begin{description}
\item[(1)] If the v-vertices of $C$ lie in distinct connected components $G_1,\ldots , G_k$ of $\G$, at least two of which are non-trivial, and none of these v-vertices are cut vertices of $\G$, then $C$ is an $\F$-diminishing cycle.
\item[(2)] If in addition,  $C$ traverses a v-vertex in every non-trivial connected component of $\G$, then  $\G\Delta C$ is a subgraph of $G$ corresponding to an Euler tour of $H$.
\end{description}
\end{cor}

\begin{proof}
{\bf (1)} Take any $i\in\{1,\ldots ,k\}$ such that $G_i$ is non-trivial, and let $v$ be a v-vertex in $V(C)\cap V(G_i)$.  Since $v$ is not a cut vertex of $\G$, we have that $G_i - v$ is connected.  It follows that $G_i\sm  E(C)$ has at most one non-trivial connected component, and Lemma~\ref{lem:dimcycle} implies that $(\G\Delta C)[V(G_1)\cup\ldots\cup V(G_k)]$ has just one non-trivial connected component.  Since $\G[V(G_1)\cup\ldots\cup V(G_k)]$ has at least two non-trivial connected components, we conclude that $C$ is an $\F$-diminishing cycle.

{\bf (2)} Now $(\G\Delta C)[V(G_1)\cup\ldots\cup V(G_k)]$ has exactly one non-trivial connected component and it contains all vertices of $\G\Delta C$ that are not isolated. Hence $\G\Delta C$ has exactly one non-trivial connected component, and it corresponds to an Euler tour of $H$ by Theorem~\ref{thm:incidencegraph}.
\end{proof}

\phantom{a}

\begin{cor}\label{cor:3comps}
Let $H$ be a covering 3-hypergraph with an Euler family $\F$ such that $\G$ has at least three connected components. Then $H$ is eulerian.
\end{cor}

\begin{proof}
First, observe that $\G$ has at most one trivial connected component, for if $u$ and $v$ are distinct v-vertices that are isolated in $\G$, then any edge of $H$ containing both $u$ and $v$ is not traversed by $\F$. Hence $\G$ has at least two non-trivial connected components.

Let $G_1,\ldots , G_k$ be the connected components of $\G$.  For each $i\in\{1,\ldots ,k\}$, choose a vertex $v_i$ in $G_i$ such that $v_i$ is not a cut vertex of $G_i$. If $G_i$ is non-trivial, then this is possible by Lemma~\ref{prop:noncut}; otherwise, the unique vertex of $G_i$ is not a cut vertex of $G_i$. 
As no two vertices among $v_1,\ldots , v_k$ lie in the same connected component of $\G$,
every e-vertex is adjacent in $\G$ to at most one of them, and hence no three of these vertices  lie in the same edge of $H$.
Therefore,  there exist e-vertices $e_1,\ldots , e_k$ such that $C=v_1e_1v_2\ldots v_ke_kv_1$ is a cycle in $G$.
Moreover, each $e_i$ is incident with exactly one $\G$-edge of $C$, so $C$ is an $\F$-interchanging cycle.
Since $C$ satisfies the conditions of Corollary~\ref{cor:dimcycle}(2), we conclude that $H$ is eulerian.
\end{proof}

\phantom{a}

\begin{cor}\label{lem:isolated} Let $H$ be a covering 3-hypergraph and  $\F$ a minimum Euler family of $H$.   If $H$ is not eulerian, then the following hold.
\begin{description}
\item[(1)] $|\F|= 2$ and $\G$ has no isolated vertices.
\item[(2)] If $C$ is an $\F$-interchanging cycle of $G$, then $\G\Delta C$ corresponds to a minimum Euler family of $H$.
\item[(3)] $G$ has no $\F$-diminishing cycle.
\end{description}
\end{cor}

\begin{proof}
{\bf (1)} By  Corollary~\ref{cor:3comps}, $\G$ has exactly two connected components, and both must be non-trivial. Hence $|\F|= 2$ and $\G$ has no isolated vertices.

{\bf (2)} By Lemma~\ref{lem:interchange}, the graph $\G\Delta C$ corresponds to an Euler family $\F'$ of $H$, and by Corollary~\ref{cor:3comps}, it has at most two connected components.  Hence $\F'$ is another minimum Euler family of $H$.

{\bf (3)} This statement follows immediately from (2).
\end{proof}

\subsection{Case $n\geq 7$}\label{sec:order7}

In this section, we prove Theorem~\ref{thm:covering} for $n\geq 7$; that is, we show that all covering 3-hypergraphs of order at least 7 are eulerian.  This case is restated below as Theorem~\ref{thm:order7}.

Lemma~\ref{lem:order7setup} will furnish our proof with most of the details needed to set up the case work.  In Theorem~\ref{thm:order7}, we will begin with a minimum Euler family $\F$ for $H$ that has two components and  minimizes the degree of an arbitrary, but fixed, vertex $v_0$ in the graph $\G$.  By Corollary~\ref{lem:isolated}(2), we need only find an $\F$-interchanging cycle $C$ --- not necessarily an $\F$-diminishing cycle --- that traverses $v_0$ and such that the Euler family corresponding to $\G \Delta C$ contradicts the assumptions on $\F$.
\\

\begin{lem}\label{lem:order7setup} Let $H$ be a covering 3-hypergraph of order $n\geq 7$ that is not eulerian.  Let $G$ be the incidence graph of $H$ and $v_0$ a v-vertex of $G$.  Let $\F$ be a minimum Euler family of $H$, and $\G$ the subgraph of $G$ corresponding to $\F$, with the property that $\deg_{\G}(v_0)$ is minimum over all minimum Euler families of $H$.

Assume that $\deg_{\G}(v_0) \ge 2$, so
there exist distinct e-vertices $e_1,e_2\in V(G)$ such that $v_0e_1, v_0e_2\in E(\G).$  Let $e_1=v_0v_1a$ and $e_2=v_0v_2b$ for some $v_1,v_2,a,b\in V(H)$ such that $v_1e_1$ and $v_2e_2$ are $\G$-edges.

Then the following hold.
\begin{description}
\item[(1)] Vertices $a$ and $b$ are distinct.
\item[(2)] Suppose there exists an e-vertex $e$ of $G$ with $e\not\in\{e_1,e_2\}$ such that $ea, eb\in E(G)$. Then both $ea$ and $eb$ are $\G$-edges.
\item[(3)] Vertices $a$ and $b$ are connected in $\G$.
\item[(4)] Let $x\in V(H)\sm \{v_0,v_1,v_2\}$ be a v-vertex disconnected from $a$ and $b$ in $\G$.  Then there exists an edge $e_x=abx$ in $H$. Any edge of $H$ containing $x$ and $a$, or $x$ and $b$, must likewise contain $a$ and $b$; and $xe_x$ is a non-$\G$-edge.
\item[(5)] Vertices $a$ and $b$ are disconnected from $v_0$ in $\G$.
\end{description}
\end{lem}

\begin{figure}[t]
\centerline{\includegraphics[scale=0.8]{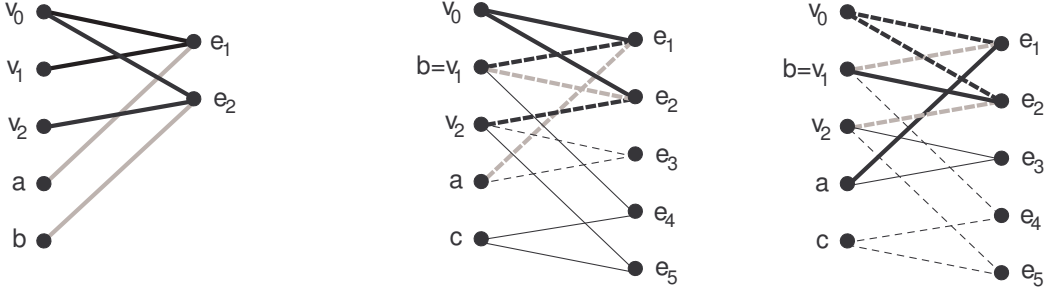}}
\caption{Relevant subgraphs of $G$ in Lemma~\ref{lem:order7setup}. The set-up (left); $\G$ and $C_1$ (centre), and $G_{\F'}=\G\Delta C_1$ and $C_2$ (right) in the proof of (3). Thick black (grey) edges are those in (not in) the subgraph corresponding to the Euler family; the status of thin edges is unknown. Dashed edges represent the interchanging cycle. Note that exactly one of the two thin edges incident with each of $e_3$, $e_4$, and $e_5$ is in $\G$ ($G_{\F'}$).}\label{fig:Lemma-1}
\end{figure}

\begin{proof}
See Figure~\ref{fig:Lemma-1} (left). Note that by Corollary~\ref{lem:isolated}(1), $\G$ has exactly two connected components, both non-trivial, and by Corollary~\ref{lem:isolated}(2) and the assumption on $\F$, there is no $\F$-interchanging cycle $C$ in $G$ such that $\deg_{\G\Delta C}(v_0) < \deg_{\G}(v_0)$.

{\bf (1)} If $a=b$, then $C=v_0e_1ae_2v_0$ is an $\F$-interchanging cycle with $\deg_{\G\Delta C}(v_0) < \deg_{\G}(v_0)$, a contradiction.

{\bf (2)} If at least one (and hence exactly one) of $ea$ and $eb$ is a non-$\G$-edge, then $C=v_0e_1aebe_2v_0$ is an $\F$-interchanging cycle with $\deg_{\G\Delta C}(v_0) < \deg_{\G}(v_0)$, a contradiction.
Therefore, both $ea$ and $eb$ are $\G$-edges, as claimed.

{\bf (3)} Suppose $a$ and $b$ are disconnected in $\G$.  Since $H$ is a covering 3-hypergraph, there exists an edge $e\in E(H)$ containing both $a$ and $b$.  By our supposition, one of $ea$ and $eb$ is necessarily a non-$\G$-edge, so (2) implies that $e=e_1$ or $e=e_2$.  Without loss of generality, assume $e=e_1$; hence $v_1=b$.

Observe that $v_0, b,$ and $v_2$ lie in the same connected component of $\G$, so $a$ and $v_2$ are disconnected in $\G$.  Let $e_3\in E(H)$ be an edge containing $a$ and $v_2$. Then exactly one of $v_2e_3$ and $ae_3$ is not in $\G$, and $C_1=ae_1be_2v_2e_3a$ is an $\F$-interchanging cycle; see Figure~\ref{fig:Lemma-1} (centre).

Define $\F'$ to be an Euler family corresponding to $\G\Delta C_1$, so $G_{\F'}=\G\Delta C_1$.  Since $H$ is not eulerian, Corollary~\ref{lem:isolated} implies that $G_{\F'}$ has exactly two connected components, both non-trivial. It is easy to verify --- see Figure~\ref{fig:Lemma-1} (right) --- that vertices $v_0$, $a$, and $b$ are mutually connected in $G_{\F'}$. Since every v-vertex is connected in $\G$ to either $a$ or $v_2$, if $v_0$ and $v_2$ are connected in $G_{\F'}$, then $G_{\F'}$ is connected, a contradiction. Thus $v_0$ and $v_2$ are disconnected in $G_{\F'}$.

By Corollary~\ref{lem:isolated}(1), vertex $a$ is not isolated in $\G$.  Let $c\neq a$ be any v-vertex connected to $a$ in $\G$.  We shall construct an $\F'$-interchanging cycle traversing $v_0$, $b$, $v_2$, and  $c$.

Let $e_4,e_5\in E(H)$ be such that $e_4$ contains $b$ and $c$, and $e_5$ contains $v_2$ and $c$.  As $e_4,e_5 \not\in \{e_1,e_2,e_3 \}$, e-vertices $e_4$ and $e_5$ have the same neighbours in $\G$ as in $G_{\F'}$. If $e_4=e_5$, then $be_4$ and $v_2e_4$ are $\G$-edges, and hence $G_{\F'}$-edges, a contradiction. Hence $e_4 \ne e_5$. Moreover, exactly one of  $be_4$ and $ce_4$, and exactly one of  $v_2e_5$ and $ce_5$, is a non-$\G$-edge.

It follows that $C_2=v_0e_1be_4ce_5v_2e_2v_0$ is an $\F'$-interchanging cycle  with $\deg_{G_{\F'}\Delta C_2}(v_0) < \deg_{\G}(v_0)$; see Figure~\ref{fig:Lemma-1} (right). Since $|\F'|=2$ and $\deg_{G_{\F'}}(v_0)=\deg_{\G}(v_0)$,  Corollary~\ref{lem:isolated}(2) yields a contradiction.

We conclude that $a$ and $b$ are connected in $\G$.

{\bf (4)} Let $x\in V(H)\sm \{v_0,v_1,v_2\}$ be a v-vertex disconnected in $\G$ from $a$ and $b$.  Since $H$ is a covering 3-hypergraph, there exist edges $e\in E(H)$ containing $a$ and $x$, and $e'\in E(H)$ containing $b$ and $x$.  If $e\neq e'$, then $C=v_0e_1aexe'be_2v_0$ is an $\F$-interchanging cycle with $\deg_{\G\Delta C}(v_0) < \deg_{\G}(v_0)$, a contradiction.  Hence $e=abx=e'$, and $ae$ and $be$ are $\G$-edges by (2), so $xe$ is a non-$\G$-edge.

\begin{figure}[t]
\centerline{\includegraphics[scale=0.8]{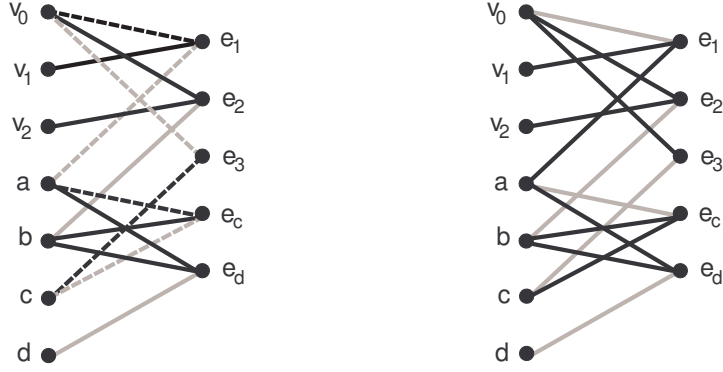}}
\caption{$\G$ and $C$ (left), and $G_{\F'}=\G\Delta C$ in the proof of Lemma~\ref{lem:order7setup}(5). Thick black (grey) edges are those in (not in) the subgraph corresponding to the Euler family, and dashed edges represent the $\F$-interchanging cycle $C$. }\label{fig:Lemma-2}
\end{figure}

{\bf (5)} Suppose that $a$ and $b$ lie in the same connected component of $\G$ as $v_0$. Call this component $G_1$, and let $G_2$ be the other connected component. By Corollary~\ref{lem:isolated}(1), component $G_2$ has at least two v-vertices, say $c$ and $d$.
By (4), there are edges $e_c=abc$ and $e_d=abd$ of $H$ such that $ce_c$ and $de_d$ are non-$\G$-edges.

Now, there exists an edge $e_3\in E(H)$ containing $v_0$ and $c$, and $e_3\not\in\{e_1,e_2\}$.  Let $C=v_0e_1ae_cce_3v_0$; see Figure~\ref{fig:Lemma-2} (left).  Then $C$ is an $\F$-interchanging cycle because $c$ and $v_0$ are disconnected in $\G$.  If $v_0e_3\in V(G_1)$, then there are two $\G$-edges of $C$ incident with $v_0$ and so $\deg_{\G\Delta C}(v_0) < \deg_{\G}(v_0)$, a contradiction.  Hence $v_0e_3$ is a non-$\G$-edge and  $e_3c \in E(G_2)$.
Since $G_2$ is an even graph and $G_2 \sm  E(C) = G_2 \sm  ce_3$, we have that $G_2\sm  E(C)$ is connected.

Since $G_1$ is an even graph, we have that $G_1\sm  v_0e_1$ is connected, and since $ae_c$ is an edge in a cycle of $G_1\sm  v_0e_1$,  we have that $(G_1\sm  v_0e_1)\sm  ae_c = G_1\sm  E(C)$ is connected; see Figure~\ref{fig:Lemma-2} (right).  Then,  by Lemma~\ref{lem:dimcycle}, $\G\Delta C$ has just one non-trivial connected component, contradicting the fact that $H$ is not eulerian.

We conclude that $a$ and $b$ are disconnected from $v_0$ in $\G$.
\end{proof}

With the setup afforded by Lemma~\ref{lem:order7setup}, we know a lot about five of the v-vertices of $\G$, namely $v_0,v_1,v_2,a,$ and $b$.  In Theorem~\ref{thm:order7}, we need only investigate the positions of two other vertices, and this will be enough to show that there must be an $\F$-interchanging cycle that contradicts our assumption that $H$ is not eulerian.\\

\newpage

\begin{thm}\label{thm:order7} Let $H$ be a covering 3-hypergraph of order $n\geq 7$.  Then $H$ is eulerian.
\end{thm}

\begin{proof} Suppose $H$ is not eulerian.   Let $G$ be the incidence graph of $H$, and fix some v-vertex $v_0$.  By Theorem~\ref{cor:quasieuleriancovering}, $H$ is quasi-eulerian, so let $\F$ be a minimum Euler family of $H$ with the property that $\deg_{\G}(v_0)$ is minimum over all minimum Euler families of $H$.  By Corollary~\ref{lem:isolated}, we know that $|\F|=2$ and that $v_0$ is not isolated in $\G$, so $\deg_{\G}(v_0)\geq 2$. Moreover, by the same corollary, $G$ has no $\F$-diminishing cycle, and no $\F$-interchanging cycle $C$ such that $\deg_{\G\Delta C}(v_0) < \deg_{\G}(v_0)$.

We have distinct e-vertices $e_1,e_2\in V(G)$ such that $v_0e_1, v_0e_2\in E(\G).$  Let $e_1=v_0v_1a$ and  $e_2=v_0v_2b$ for some v-vertices $v_1$, $a$, $v_2$, and $b$.
Without loss of generality, let $v_1e_1$ and $v_2e_2$ be $\G$-edges, and $ae_1$ and $be_2$ be non-$\G$-edges.
Now $H$ satisfies the assumptions of Lemma~\ref{lem:order7setup} with exactly the same set-up.  We have that $v_0$, $v_1,$ and $v_2$ lie in one connected component, $G_1$, of $\G$, and $a$ and $b$ lie in the other connected component, $G_2$, of $\G$.  Note that $a\neq b$, but it is possible that $v_1=v_2$.

Since $|V(H)|\geq 7$, we must have at least two v-vertices in $\G$ besides $v_0,v_1,v_2,a$, and $b$.  We shall complete the proof by splitting into cases depending on where two of these additional v-vertices lie.

{\sc Case 1: there exist v-vertices $c\in V(G_1)\sm  \{v_0,v_1,v_2\}$ and $d\in V(G_2)\sm  \{a,b\}$.}

Since $c$ is disconnected in $\G$ from $a$ and $b$, by Lemma~\ref{lem:order7setup}(4), there exists an edge $e_c=abc$.

Let $e_3\in E(H)$ be an edge containing $a$ and $d$, and suppose $ae_3$ or $de_3$ is a non-$\G$-edge.  There exists an edge $f\in E(H)$ containing $c$ and $d$, and $f\neq e_3$ because $e_3$ must contain three v-vertices of $G_2$.  Hence $C=v_0e_1ae_3dfce_cbe_2v_0$ is an $\F$-interchanging cycle with $\deg_{\G\Delta C}(v_0) < \deg_{\G}(v_0)$, a contradiction.  Hence $ae_3$ and $de_3$ are $\G$-edges.

Similarly, let $e_4\in E(H)$ be an edge containing $b$ and $d$, and conclude that $be_4$ and $de_4$ are $\G$-edges.  (Note that this also implies that $e_3\neq e_4$.)

Let $x$ be the v-vertex such that $e_3=adx$.  If $x=b$, then $C=v_0e_1ae_3be_2v_0$ is an $\F$-interchanging cycle with $\deg_{\G\Delta C}(v_0) < \deg_{\G}(v_0)$, a contradiction.  Hence $x\neq b$.

If $x\in V(G_2)$ and $f$ is an edge containing $x$ and $c$, then $C=v_0e_1ae_3xfce_cbe_2v_0$ is an $\F$-interchanging cycle with $\deg_{\G\Delta C}(v_0) < \deg_{\G}(v_0)$, a contradiction.  Hence $x\in V(G_1)$.

If $x\in \{v_0,v_1\}$, then $C=xe_1ae_3x$ is an $\F$-interchanging cycle with just one $\G$-edge in each of $G_1$ and $G_2$.  By Corollary~\ref{cor:simpledimcycle}, $C$ is an $\F$-diminishing cycle, a contradiction. Thus $x\not\in \{v_0,v_1\}$.

If $x\in V(G_1)\sm \{v_0,v_1,v_2\}$, then by Lemma~\ref{lem:order7setup}(4), there exists an edge $e_x=abx$, and $xe_x$ is a non-$\G$-edge.  Then $C=v_0e_1ae_3xe_xbe_2v_0$ is an $\F$-interchanging cycle with $\deg_{\G\Delta C}(v_0) < \deg_{\G}(v_0)$, a contradiction.

We thus conclude that $x=v_2$ (and $v_1\neq v_2$), so $e_3=adv_2$.  By an analogous argument (swapping the roles of $a$ and $b$, of $v_1$ and $v_2$, and of $e_1$ and $e_2$), we have that $e_4=bdv_1$.

\begin{figure}[t]
\centerline{\includegraphics[scale=0.8]{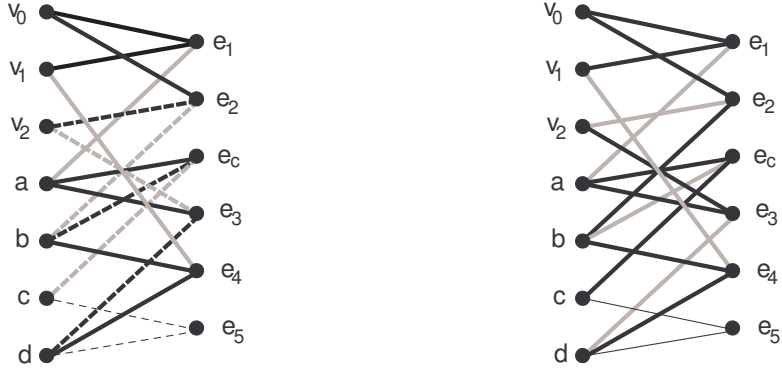}}
\caption{$\G$ (left) and $G_{\F'}=\G\Delta C$ (right) in Case 1. Thick black (grey) edges are those in (not in) the subgraph corresponding to the Euler family, and dashed edges represent the $\F$-interchanging cycle $C$. Note that exactly one of the two thin edges incident with $e_5$ is in $\G$ ($G_{\F'}$).}\label{fig:Case1}
\end{figure}

Let $e_5$ be an edge containing $c$ and $d$, and note that $e_5\not\in \{e_2,e_c,e_3\}$.  Then $C=be_cce_5de_3v_2e_2b$. is an $\F$-interchanging cycle; see Figure~\ref{fig:Case1} (left). Let $\F'$ be an Euler family corresponding to $G_{\F'}=\G\Delta C$.  By Lemma~\ref{lem:isolated}, $G_{\F'}$ has exactly two connected components, both non-trivial, and it can be verified easily --- see Figure~\ref{fig:Case1} (right) --- that v-vertices $v_0$, $v_1$, $b$, and $d$ are mutually connected in $G_{\F'}$, as are $v_2$, $a$, and $c$. Since every vertex is connected in $\G$, and hence in $G_{\F'}$, to one of these seven v-vertices, we conclude that
$v_0$, $v_1$, $b$, and $d$ lie in one connected component of $G_{\F'}$, and $v_2,a,$ and $c$ lie in the other.

Observe that, since $G_{\F'}$ satisfies the initial assumptions made of $\G$, we may apply Lemma~\ref{lem:order7setup} to $G_{\F'}$, with the roles of $v_2$ and $b$ swapped.  In addition, what we know so far on Case 1 applies to $G_{\F'}$ as well, with the roles of $v_2$ and $b$, of $c$ and $d$, and of $e_c$ and $e_3$ swapped.

In particular, we deduce that there exists an e-vertex $e_6=v_1v_2c$ that plays the same role in $G_{\F'}$ as $e_4=v_1bd$ does in $\G$.  Let $C'=v_0e_1ae_cbe_4v_1e_6v_2e_2v_0$, and observe that $C'$ is an $\F'$-interchanging cycle with $\deg_{G_{\F'}\Delta C'}(v_0) < \deg_{G_{\F'}}(v_0)$, a contradiction.

{\sc Case 2: $a$ and $b$ are the only v-vertices in $G_2$.}

Let $X=V(H)\sm \{v_0,v_1,v_2,a,b\}$, so $|X|\geq 2$.

Suppose first that $v_1 \ne v_2$. Let $e$ be an edge of $H$ containing $v_2$ and $a$.

If $e=v_0v_2a$, then $C=v_0e_1aev_0$ is an $\F$-interchanging cycle. If $v_0e \not\in E(G_1)$, then $C$ is an $\F$-diminishing cycle by Corollary~\ref{cor:simpledimcycle}. Otherwise, $\deg_{G_{\F}\Delta C}(v_0) < \deg_{G_{\F}}(v_0)$, so we have  a contradiction in both cases.

If $e=v_2ab$, then $C=v_2e_2bev_2$ is an $\F$-diminishing cycle by Corollary~\ref{cor:simpledimcycle}, a contradiction.

If $e=v_2ax$, for some $x \in X$, then by Lemma~\ref{lem:order7setup}(4), we have $b \in e$,  a contradiction.

Hence $e=v_1v_2a$, and by symmetry, there is an edge $e'=v_1v_2b$. Thus $v_1ev_2e'v_1$ is a cycle in $G_1$, and $C=v_1e_1aev_1$ is an $\F$-interchanging cycle. Since $G_1 \sm E(C)=(G_1 \sm v_1e_1) \sm v_1e$ and $v_1e$ lies in a cycle of $G_1 \sm v_1e_1$, which is connected,  $\G \Delta C$ is connected by Lemma~\ref{lem:dimcycle} --- a contradiction.

Hence $v_1=v_2$, and $v_0e_1v_1e_2v_0$ is a cycle in $G_1$. Fix some $x \in X$. By Lemma~\ref{lem:order7setup}(4), there is an edge $e_x=abx$ in $H$, and $xe_x$ is a non-$\G$-edge. Let $e$ be an edge of $H$ containing $v_1$ and $x$. By Lemma~\ref{lem:order7setup}(4), we know $a,b \not\in e$.

Suppose $e=v_1xy$ for some $y \in X \sm \{x\}$.
By Lemma~\ref{lem:order7setup}(4), there is an edge $e_y=aby$ in $H$, and $ye_y$ is a non-$\G$-edge.

If $xe, ye \in E(G_1)$, then $C=v_1e_1ae_xxev_1$ is an $\F$-interchanging cycle. Since $v_1e_1$ lies in a cycle of $G_1 \sm xe$, we have that  $G_1\sm E(C)=(G_1 \sm xe) \sm v_1e_1$ is connected.
Since $G_2 \sm E(C)$ is connected as well, $\G \Delta C$ is connected by Lemma~\ref{lem:dimcycle} --- a contradiction.

Hence $v_1e \in E(G_1)$. If $xe \not\in E(G_1)$, then let $C=v_1e_2be_xxev_1$; otherwise, let $C=v_1e_2be_yyev_1$. In either case, $C$ is an $\F$-interchanging cycle, and it can be seen as above that both $G_1\sm E(C)$ and $G_2\sm E(C)$ are connected, a contradiction.

We conclude that $e=v_0v_1x$. If $v_1e, v_1e \in E(G_1)$, then $C=v_0e_1ae_xxev_0$ is an $\F$-interchanging cycle with $\deg_{G_{\F}\Delta C}(v_0) < \deg_{G_{\F}}(v_0)$ ---  a contradiction.

Hence $xe \in E(G_1)$. If $v_0e \not\in E(G_1)$, then let $C=v_0exe_xae_1v_0$; otherwise, let $C=v_1exe_xae_1v_1$. In either case, $C$ is an $\F$-interchanging cycle, and it can be seen as above that both $G_1\sm E(C)$ and $G_2\sm E(C)$ are connected, a contradiction.

{\sc Case 3: $v_0, v_1,$ and $v_2$ are the only v-vertices of $G_1$.}

Let $e_x=abx$ be an edge of $H$ containing $a$, $b$, and some vertex $x$.  Since $e_x \not\in \{ e_1, e_2 \}$, by Lemma~\ref{lem:order7setup}(2), we have that $ae_x$ and $be_x$ are $\G$-edges, and hence $xe_x$ is a non-$\G$-edge.

If $x\in\{v_0,v_1\}$, then $C=xe_1ae_xx$ is an $\F$-diminishing cycle by Corollory~\ref{cor:simpledimcycle}, a contradiction. Similarly, if $x=v_2$, then $C=v_2e_2be_xv_2$ is an $\F$-diminishing cycle.

Hence $x\in V(G_2)$.
Since, by Corollary~\ref{cor:noncut}, the graph $G_1$ has at least two v-vertices that are not cut vertices, we may assume there exists $y\in\{v_0,v_1\}$ that is not a cut vertex of $G_1$.  Let $e_3=xyz$ for some $z \in V(H)$.
Then $C=ye_1ae_xxe_3y$ is an $\F$-interchanging cycle.  Let $\F'$ be an Euler family corresponding to $G_{\F'}=\G\Delta C$.

\begin{figure}[t]
\centerline{\includegraphics[scale=0.8]{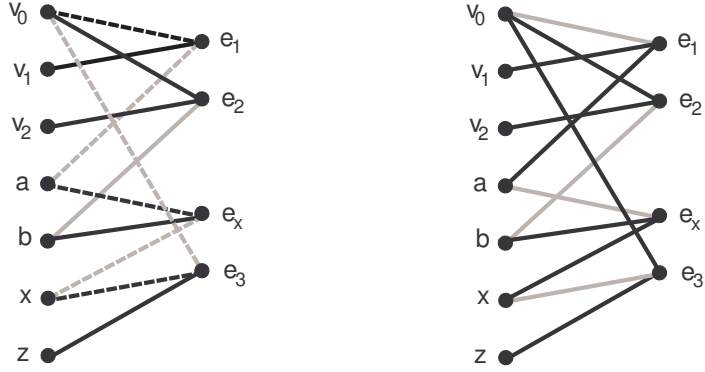}}
\caption{$\G$ (left) and $G_{\F'}=\G\Delta C$ (right) in Case 3, for $y=v_0$. Thick black (grey) edges are those in (not in) the subgraph corresponding to the Euler family, and dashed edges represent the $\F$-interchanging cycle $C$. }\label{fig:Case3a}
\end{figure}

Suppose $y=v_0$.  If $v_0e_3$ is a $\G$-edge, then $\deg_{G_{\F'}}(v_0) < \deg_{\G}(v_0)$, a contradiction.  Hence $v_0e_3$ is non-$\G$-edge --- see Figure~\ref{fig:Case3a} (left) --- and  $G_1\sm  E(C)=G_1 \sm v_0e_1$ is connected.
By Corollary~\ref{lem:isolated}(2),  $\F'$ is a minimum Euler family that minimizes $\deg_{G_{\F'}}(v_0) = \deg_{\G}(v_0)$, so it satisfies our assumptions on $\F$. Moreover, the set-up of Lemma~\ref{lem:order7setup} applies to $\F'$, with $v_0,v_1,v_2,a,b,e_1,e_2$ replaced by $v_0,z,v_2,x,b,e_3,e_2$, respectively; see Figure~\ref{fig:Case3a} (right). Since the connected component of $G_{\F'}$ containing $v_0$ also contains vertices $v_1$ and $a$, we can see that $G_{\F'}$ falls under Case 1 or Case 2 of this proof, both of which lead to a contradiction.

\begin{figure}[t]
\centerline{\includegraphics[scale=0.8]{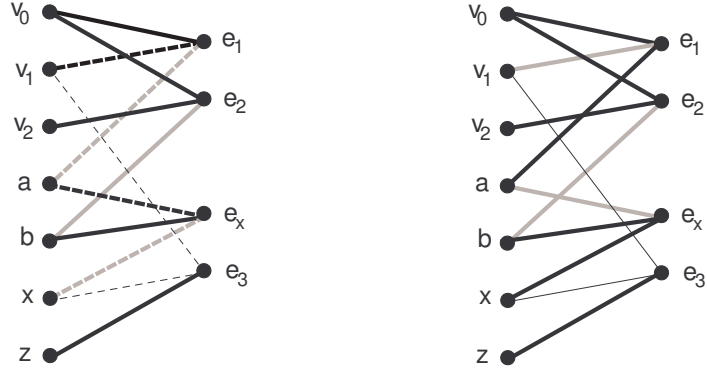}}
\caption{$\G$ (left) and $G_{\F'}=\G\Delta C$ (right) in Case 3, for $y=v_1$. Thick black (grey) edges are those in (not in) the subgraph corresponding to the Euler family, and dashed edges represent the $\F$-interchanging cycle $C$. Note that exactly one of the two thin edges incident with $e_3$ is in $\G$ ($G_{\F'}$).}\label{fig:Case3b}
\end{figure}

Hence $y=v_1$; see Figure~\ref{fig:Case3b} (left). Suppose first that $v_1$ and $b$ are connected in $G_{\F'}$.  Since $v_1$ is not a cut vertex of $G_1$,  either $G_1\sm  E(C)$ is connected, or $(G_1\sm  E(C)) - v_1$ is connected and $v_1$ is isolated in $G_1\sm  E(C)$.  If $v_1$ is isolated in $G_1\sm  E(C)$, then  it is isolated in $G_{\F'}$ as well, a contradiction.  Hence $G_1\sm  E(C)$ is connected. It is then easy to verify that  $v_0$, $v_1$, $v_2$, $a$, $b$, and $x$ lie in the same connected component of $G_{\F'}$; see Figure~\ref{fig:Case3b} (right).  Since all other v-vertices are connected to one of these in $G_{\F'}$, we see that $G_{\F'}$ is connected, a contradiction.

Hence $v_1$ and $b$ are disconnected in $G_{\F'}$.  Apply Lemma~\ref{lem:order7setup}(2) to $G_{\F'}$ with the roles of $v_1$ and $a$ swapped. Thus $v_1$ and $b$ are connected in $G_{\F'}$, a contradiction.

Having reached a contradiction in all cases, we conclude that $H$ is eulerian.
\end{proof}

\subsection{Case $3\leq n\leq 6$}\label{sec:order6}

In this section, we prove Theorem~\ref{thm:covering} for the orders that were missed by Theorem~\ref{thm:order7}. \\

\begin{lem}\label{lem:coveringorder3-5} Let $H$ be a covering 3-hypergraph of order $n$, for $3 \le n \le 6$, and with at least two edges.  Then $H$ is eulerian.
\end{lem}

\begin{proof} If $n=3$, then  $H$ is a triple system with at least two edges, and hence eulerian by \cite{SW}. Hence assume $n\in \{ 4,5,6 \}$.

Let $G$ be the incidence graph of $H$.  By Theorem~\ref{cor:quasieuleriancovering}, $H$ is quasi-eulerian, so let $\F$ be a minimum Euler family for $H$, with the corresponding subgraph $\G$ of $G$.  Suppose $H$ is not eulerian; then Corollary~\ref{lem:isolated} implies that $\G$ has exactly two connected components, both non-trivial, and admits no $\F$-diminishing cycle. Let $G_1$ be the connected component with fewer v-vertices.

{\sc Case 1: $G_1$ has three v-vertices.}

Hence $n=6$. Let $U_1=\{a,b,c\}$ and $U_2=\{u,v,w \}$ be the sets of  v-vertices of $G_1$ and $G_2$, respectively.   Since each of the nine pairs $\{x_1,x_2 \}$, for $x_1\in U_1$ and $x_2\in U_2$ is contained in some edge of $H$, and every edge of $H$ contains an even number of such pairs, at least one pair must be contained in two edges.

Without loss of generality, assume that $a,u\in e_1 \cap e_2$ for distinct edges $e_1,e_2\in E(H)$.  Then $C=ae_1ue_2a$ is an $\F$-interchanging cycle of $G$. If $e_1 \in V(G_1)$ and $e_2 \in V(G_2)$, then $C$ satisfies the assumptions of Corollary~\ref{cor:simpledimcycle}, and so it is an $\F$-diminishing cycle, a contradiction.

Hence, without loss of generality, $e_1,e_2 \in V(G_1)$. It follows that $ue_1$ and $ue_2$ are non-$\G$-edges, and $G_2\sm E(C)=G_2$ is connected.

If $a$ is not a cut vertex of $G_1$, then $G_1 \sm E(C)$ is either connected, or else has a single non-trivial connected component and an isolated vertex, $a$. In either case, by Lemma~\ref{lem:dimcycle},  $\G \Delta C$ has a single non-trivial connected component, a contradiction.

Hence $a$ is a cut vertex of $G_1$, and by Corollary~\ref{cor:noncut}, neither $b$ nor $c$ is. If $e_1$ and $e_2$ both contain $b$ or both contain $c$, then a contradiction is obtained as in the previous paragraph, by replacing $a$ in $C$ with $b$ or $c$, respectively. Since $e_1, e_2 \in V(G_1)$, we may thus assume $b \in e_1$ and $c \in e_2$, so that $be_1ae_2c$ is a path in $G_1$. Since $a$ is a cut vertex, by Corollary~\ref{cor:noncut}, there is no cycle in $G_1$ containing all of $a$, $b$, and $c$. Therefore, as $G_1$ is an even graph, it has e-vertices $e_3$ and $e_4$ such that $ae_1be_3a$ and $ae_2ce_4a$ are cycles in $G_1$. But then $G_1 \sm E(C)$ is connected, yielding a contradiction as above.

{\sc Case 2: $G_1$ has exactly two v-vertices.}

Let $m$ be the number of v-vertices in $G_2$, and denote the v-vertices of $G_1$ by $a$ and $b$.  Since $G_1$ is an even graph, it has an even number of e-vertices, each of the form $e=abx$ for some v-vertex $x$ of $G_2$.

If $e_1$ and $e_2$ are two e-vertices of $G_1$ such that $e_1=abx=e_2$, then $C=ae_1xe_2a$ is an $\F$-interchanging cycle. Moreover, it is not difficult to verify that $\G\Delta C$ has just one non-trivial connected component,  a contradiction.

It follows that the number of e-vertices of $G_1$ is no larger than $m$, the number of v-vertices in $G_2$.

Let $e_1=abu$ and $e_2=abv$, for distinct v-vertices $u$ and $v$ of $G_2$, be two e-vertices of $G_1$. If $G_2$ has an e-vertex $e$ containing $a$ and $u$, then $C=ae_1uea$ is an $\F$-interchanging cycle with exactly one $\G$-edge in each connected component of $\G$, so by Corollary~\ref{cor:simpledimcycle}, it is an $\F$-diminishing cycle, a contradiction. A similar contradiction is obtained if $G_2$ has an e-vertex $e$ containing  one of $a$ and $b$, and one of $u$ and $v$.
If $m=2$, then $e$ can be taken to be an edge containing $u$ and $v$, and if $m=3$, then $e$ can be taken to be an edge containing $a$ and $w$, the third v-vertex in $G_2$.

Hence we may assume that $m=4$, and $G_2$ has no e-vertex containing one of $a$ and $b$, and one of $u$ and $v$.
Label the remaining v-vertices of $G_2$ by $w$ and $z$, and observe that $G_1$ has either exactly two or exactly four e-vertices.

Suppose first that $G_1$ has exactly four e-vertices; namely e-vertices $e_3=abw$ and $e_4=abz$, in addition to $e_1=abu$ and $e_2=abv$. Let $e_5$ be an edge of $H$ containing $u$ and $v$. Without loss of generality, $e_5=uvw$, and $ue_5$ and $we_5$ are $\G$-edges. Then $C=ae_1ue_5ve_2a$ is an $\F$-interchanging cycle, and it is not difficult to verify that $G_1\sm E(C)$ and $G_2\sm E(C)$ are connected. Hence by Lemma~\ref{lem:dimcycle}, $C$ is an $\F$-diminishing cycle, a contradiction.

Hence $G_1$ has exactly two e-vertices. Since $e_1$ and $e_2$ are the only edges of $H$ containing one of $a$ and $b$, and one of $u$ and $v$, any edge containing $a$ and $w$, or $b$ and $w$, must also contain $z$. Hence we have edges $e_3=awz$ and $e_4=bwz$, and $we_3, we_4, ze_3, ze_4 \in E(G_2)$. Let $e_5$ be an edge containing $u$ and $v$. Without loss of generality, $e_5=uvw$ and $ue_5 \in E(G_2)$.

If $we_5 \in E(G_2)$, let $C=be_2ve_5we_4b$, and if $ve_5 \in E(G_2)$, let $C=ae_1ue_5we_3a$. In either case $C$
is an $\F$-interchanging cycle, and it is straightforward to verify that $\G \Delta C$ is connected, a contradiction.

Having reached a contradiction in all cases, we conclude that $H$ is eulerian.
\end{proof}

\small

\section*{Acknowledgements}

The first author gratefully acknowledges support by the Natural Sciences and Engineering Research Council of Canada (NSERC), Discovery Grant RGPIN-2016-04798.

\end{document}